\documentclass[10pt]{amsart}
\usepackage{qtree,array,youngtab}
\usepackage{color, graphicx, comment}
\usepackage{tikz}
\usepackage{amsmath}

\definecolor{darkgreen}{rgb}{0.,0.5,0.}

 \usepackage{geometry}
\newcommand{\la}{\lambda}

\allowdisplaybreaks

\numberwithin{equation}{section} \overfullrule 5pt
\newtheorem{thm}{Theorem}[section]
\newtheorem{cor}[thm]{Corollary}

\newtheorem{lem}[thm]{Lemma}

\newtheorem{ex}[thm]{Example}
\newtheorem{defi}[thm]{Definition}

\title[Polynomiality of certain average weights for oscillating tableaux]{Polynomiality of certain average weights for oscillating tableaux}
\date{September 1, 2017}
\author{Guo-Niu HAN and Huan XIONG}
\address{Universit\'e de Strasbourg, CNRS, IRMA UMR 7501, F-67000 Strasbourg, France}
\email{guoniu.han@unistra.fr, \quad xiong@math.unistra.fr}

\subjclass[2010]{05A15, 05A17, 05A19, 11P81}

\keywords{oscillating tableau, partition, Young diagram, Young's lattice}

\begin{document}
\begin{abstract}
We prove that  a family of average weights for oscillating tableaux are polynomials in two variables, namely, the length of the oscillating tableau and the size of the ending partition, which generalizes a result of Hopkins and Zhang. Several explicit and asymptotic formulas for the average weights are also derived.  
\end{abstract}

% >>>
\maketitle

% BODY

\section{introduction}
The basic knowledge on partitions and Young's lattice can be found in \cite{Macdonald,stanley2}.
A \emph{partition} is a finite weakly decreasing sequence of  positive
integers $\lambda = (a_1, a_2, \ldots, a_L)$, where
 $a_i \ (1\leq i\leq L)$ are called the parts of $\la$.
The integer  $|\lambda|=\sum_{1\leq i\leq L}a_i$ is called
the \emph{size} of  $\la$.  
Let $\mathbb{P}$ be the set of all partitions. The partition~$\lambda$ is identical with 
its \emph{Young diagram}, which is a collection of boxes arranged in left-justified rows with $a_i$ boxes in the $i$-th row.
A \emph{standard Young tableau}  of shape~$\la$ is a filling in the
boxes of the Young diagram of $\la$ with distinct numbers from $1$ to $|\la|$ such
that the numbers in each row and each column are increasing (see Figure~\ref{fig:2.1}).
 Equivalently,  a standard Young tableau of shape $\lambda$ can be seen as a sequence of partitions $T = (\lambda^{0},\lambda^{1},\ldots,\lambda^{l})$ such that $\lambda^{0} = \emptyset$, $\lambda^{l} = \lambda$; and $\la^{i+1}$ is obtained by adding a box to  $\lambda^i$ for $0\leq i \leq l-1$.  Denote by $f_\lambda$ the number of standard Young tableaux of shape
$\lambda$. In $1954$, Frame, Robinson and Thrall \cite{FRT} proved the following celebrated hook length formula, which shows that the number of standard Young tableaux of shape
$\lambda$ is determined by hook lengths of $\la$:
\begin{align}
f_\lambda = \frac{|\lambda| !}{H_{\lambda}},
\end{align}
where $H_{\lambda}$ is the product
of all hook lengths of boxes in the Young diagram of $\lambda$. Various proofs of the above hook length formula were given in \cite{bandlow,GNW,Kratt1995,NPS}.
\begin{figure}[htbp]
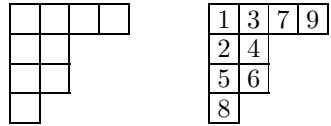

\begin{center}
\Yvcentermath1

\begin{tabular}{c}

$\yng(4,2,2,1)$
\ \ \ \ \ \ \ \
$\young(1379,24,56,8)$

\end{tabular}

\end{center}
\caption{The Young diagram of the partition $(4,2,2,1)$ and a standard Young tableau of shape  $(4,2,2,1)$.}
\label{fig:2.1}
\end{figure}

Oscillating tableaux are generalizations of standard Young tableaux. An \emph{oscillating tableau of shape $\lambda$ and length~$l$} is a sequence of partitions $T = (\lambda^{0},\lambda^{1},\ldots,\lambda^{l})$ such that $\lambda^{0} = \emptyset$, $\lambda^{l} = \lambda$; and $|\la^i/\la^{i+1}|=1$ or $|\la^{i+1}/\la^{i}|=1$ for each $0\leq i \leq l-1$, i.e., $\la^{i+1}$ is obtained by adding a box to or removing a box from $\lambda^i$. Therefore the oscillating tableau can be seen as a walk from $\emptyset$ to~$\la$ in Young's lattice (see \cite{stanley3}).  
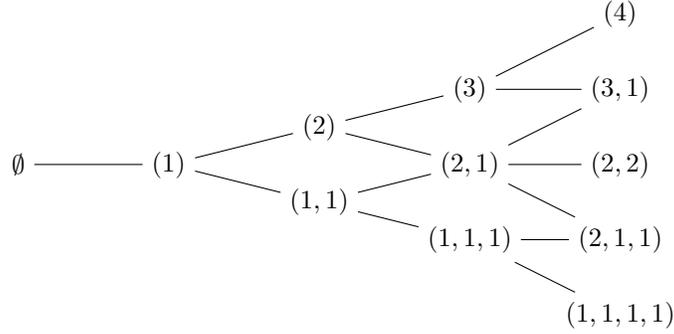
\begin{figure}

\begin{tikzpicture} 
  \node (ba) at (2,4.5) {$(2)$};
  \node (bb) at (2,3.5) {$(1,1)$};
  \node (ca) at (4,5) {$(3)$};
  \node (cb) at (4,4) {$(2,1)$};
  \node (cc) at (4,3) {$(1,1,1)$};
  
  \node (da) at (6,6) {$(4)$};
  \node (db) at (6,5) {$(3,1)$};
  \node (dc) at (6,4) {$(2,2)$};
  \node (dd) at (6,3) {$(2,1,1)$};
  \node (de) at (6,2) {$(1,1,1,1)$};

  \node (aa) at (0,4) {$(1)$};
  \node (min) at (-2,4) {$\emptyset$};
  \draw (min) -- (aa) -- (ba) -- (ca) -- (da); 
  \draw (aa) -- (bb) -- (cb)  -- (db);
  \draw (bb) -- (cc)  -- (dd);
  \draw (cc) -- (de);
  \draw (cb) -- (dc);
  \draw (cb) -- (dd);
  \draw (ca) -- (db);
  \draw (ba) -- (cb);

\end{tikzpicture}
\caption{The Young's lattice of partitions with sizes at most 4.}
\label{fig:2.2}
\end{figure}
Let $\mathcal{OT}(\lambda,l)$  be the set of oscillating tableaux of shape $\lambda$ and length~$l$.  The cardinality of $\mathcal{OT}(\lambda,l)$ is well-known as following.

\begin{thm}[\cite{roby,stanley1,stanley3,sundaram}]\label{th:r=0}
Let $\la\in \mathbb{P}$  and $k=|\lambda|$. 
Then for all $n \in \mathbb{N}$, we have
$$\#\mathcal{OT}(\lambda,k+2n) = \binom{k+2n}{k}\, (2n-1)!!\, f_{\lambda}.$$
On the other hand,  $\#\mathcal{OT}(\lambda,l) = 0$ if $l \neq k + 2n$ for any $n \in \mathbb{N}$.
\end{thm}

Bijective proofs of Theorem \ref{th:r=0} were given in \cite{roby,sundaram}. Another proof was obtained by Stanley \cite{stanley1,stanley3} in the study of differential posets.  The enumerations of various oscillating tableaux with restrictive conditions can be found in \cite{Bur2016,Bur2014,chen2,Kratt2016,Kratt1996,pak}.
In $2015$, Hopkins and Zhang \cite{HZ14} proved the following result on the average of certain weight function of oscillating tableaux. Their proof is motivated by Stanley's  theory of differential posets \cite{stanley1,stanley3}.  
\begin{thm}[\cite{HZ14}]\label{thm:HZ14}
Let $\la\in \mathbb{P}$  and $k=|\lambda|$. 
Then for all $n \in \mathbb{N}$,
$$
	\frac{1}{\# \mathcal{OT}(\lambda,k+2n)} \sum_{T \in \mathcal{OT}(\lambda,k+2n)} \mathrm{wt}(T) = (k+2n+1) \cdot \frac {3k+2n}{6},
$$
where ~$\mathrm{wt}(T) := \sum_{i=0}^{l} |\lambda^{i}|$ for each oscillating tableau $T = (\lambda^{0},\ldots,\lambda^{l})$.
\end{thm}

As noted by Hopkins and Zhang \cite{HZ14}, it is surprising that the above average in Theorem ~\ref{thm:HZ14}  is a polynomial of $n$ and $|\la|$.  
In this paper, we generalize Theorem  \ref{thm:HZ14} and show that this polynomiality holds for a family of weight functions for oscillating tableaux. The following is our main result.
%Let $\mathrm{wt}_r(T) := \sum_{i=0}^{l} |\lambda^{i}|^r $ for any oscillating tableau $T = (\lambda^{0},\ldots,\lambda^{l})$. 
%We obtain the following result, which  generalizes Theorems \ref{th:r=0} and \ref{thm:HZ14}.
\begin{thm} \label{th: main_size}
Let $P(x,y)$ be a given polynomial of two variables $x$ and $y$. For each oscillating tableau $T = (\lambda^{0},\ldots,\lambda^{l})$, let
$\mathrm{wt}_P(T) := \sum_{i=0}^{l}P(|\la^i|,i)$.
Then there exists a polynomial $Q(x,y)$ with the same degree and constant term as $P(x,y)$ such that
\begin{align}\label{eq: Main_size}
\frac1{\# \mathcal{OT}(\lambda,|\lambda|+2n)} \sum_{T \in \mathcal{OT}(\lambda,|\lambda|+2n)} \mathrm{wt}_P(T) =  (|\la|+2n+1)\, Q(|\la|,|\la|+2n)
\end{align} 
for any $n\in \mathbb{N}$ and $\la\in \mathbb{P}$.
\end{thm}
Theorem \ref{th: main_size} tells us that the average of the weight function $\mathrm{wt}_P(T)$ is a polynomial of $n$ and $|\la|$ with degree $\deg(P)+1$.  
In fact, the polynomial $Q(x,y)$ in Theorem \ref{th: main_size} equals
$\Psi^{-1}\bigl(P(x,x+2y)\bigr)$, where $\Psi$ is an operator given in Definition~\ref{def:psi}. Usually, it is not easy to compute the explicit formula for $Q(x,y)$ when $P(x,y)$ is given except for some special cases that we will describe below. 
For example, let $\la=\emptyset$ and $P(x,y)=\binom xr $ where $r\in\mathbb{N}$, 
we can derive the following explicit formulas. 
\begin{cor}\label{th:empty_partition}
For any $n,r\in \mathbb{N}$, we have 
\begin{align}\label{eq:size_formula}
\frac{2^n\, n!}{(2n+1)!}\ \sum_{ (\lambda^{0},\la^1,\ldots,\lambda^{2n}) \in \mathcal{OT}(\emptyset,2n)} \ 
\sum_{i=0}^{2n}\binom{|\la^i|}{r}
=\frac{2^r\, r!^2}{(2r+1)!} \binom{n}{r}.
\end{align} 
\end{cor}
Also, by letting  $P(x,y)=x$  in Theorem \ref{th: main_size}, we derive  Theorem \ref{thm:HZ14}.  For each oscillating tableau $T = (\lambda^{0},\ldots,\lambda^{l})$, let
$$
\mathrm{wt}_{a,b}(T) := \sum_{i=0}^{l}|\la^i|^a\cdot i^b
$$  
where $a,b\in \mathbb{N}$. Then, we obtain the following corollaries for the weight functions $P(x,y)=x^2$ and $xy$ respectively.
\begin{cor} \label{th:r=2}
Let $\la\in \mathbb{P}$  and $k=|\lambda|$.  Then for any $n \in \mathbb{N}$, we have
$$  
\frac{1}{\# \mathcal{OT}(\lambda,k+2n)} \sum_{T \in \mathcal{OT}(\lambda,k+2n)} \mathrm{wt}_{2,0}(T) = (k+2n+1)\cdot\frac{10k^2+4n^2+10kn+5k+6n}{30}.
$$
\end{cor}

\begin{cor} \label{th:xy}
Let $\la\in \mathbb{P}$  and $k=|\lambda|$.  Then for any $n \in \mathbb{N}$, 
$$  
\frac{1}{\# \mathcal{OT}(\lambda,k+2n)} \sum_{T \in \mathcal{OT}(\lambda,k+2n)} \mathrm{wt}_{1,1}(T) = (k+2n+1)\cdot\frac{2k^2+2n^2+5kn+k}{6}.
$$
\end{cor}

For the general weight function $\mathrm{wt}_{i,j}(T)$, we can derive the following asymptotic formulas of their averages. 

\begin{thm} \label{th: asy}
Let  $i,j\in\mathbb{N}$. For a fixed nonnegative integer $n$, we have  
\begin{align}\label{eq: asy1}
\frac1{\# \mathcal{OT}(\lambda,|\lambda|+2n)} \sum_{T \in \mathcal{OT}(\lambda,|\lambda|+2n)} \mathrm{wt}_{i,j}(T)\ \sim\  \frac {|\la|^{i+j+1}}{i+j+1}
\end{align} 
when $|\la|\rightarrow \infty$.

For a fixed partition $\la$, we have 
\begin{align}\label{eq: asy2}
\frac1{\# \mathcal{OT}(\lambda,|\lambda|+2n)} \sum_{T \in \mathcal{OT}(\lambda,|\lambda|+2n)} \mathrm{wt}_{i,j}(T)\ \sim\  \frac {i!(i+j)!(2n)^{i+j+1}}{(2i+j+1)!}
\end{align} 
when $n\rightarrow \infty$.
\end{thm}

The main idea in the proof of our results is to  study the operator~$\Psi$ given in Definition \ref{def:psi}, which is shown to be a bijection  from the set of real coefficient polynomials with two parameters to itself.

\section{Proofs of main results} \label{sec:size}
In this section, we prove Theorem \ref{th: main_size} and the three corollaries stated in the introduction.
Let $\la$ be a partition. Denote by $\Omega^+(\la)$ (resp. $\Omega^-(\la)$) the set of partitions $\lambda^{+}$ (resp. $\lambda^-$)  obtained by adding (resp. removing) a box to (resp. from) $\lambda$:
$$
\Omega^+(\la):=\{  \la^+: |\la^+/\la|=1   \}
$$
and 
$$
\Omega^-(\la):=\{  \la^-: |\la/\la^-|=1   \}.
$$

 \medskip
 
\begin{ex} For the partition $\la=(5,2,2,1)$, we have 
$$
\Omega^+(\la):=\{  (6,2,2,1),  (5,3,2,1), (5,2,2,2),(5,2,2,1,1)   \}
$$
and 
$$
\Omega^-(\la):=\{  (4,2,2,1), (5,2,1,1), (5,2,2)   \}.
$$
\end{ex}

We need the following lemma. 
\begin{lem} \label{th:la+-}
Let $\la\in \mathbb{P}$  and $k=|\lambda|$.  For any $n\in \mathbb{N}$ with $k+2n>0$ we have
$$
\frac{\sum_{\lambda^+\in\ \Omega^+(\la)}\#\mathcal{OT}(\lambda^+,k+2n-1)}{\#\mathcal{OT}(\lambda,k+2n)}=\frac{2n}{k+2n}
$$
and
$$
\frac{\sum_{\lambda^- \in\ \Omega^-(\la)}\#\mathcal{OT}(\lambda^-,k+2n-1)}{\#\mathcal{OT}(\lambda,k+2n)}=\frac{k}{k+2n}.
$$
\end{lem}
\begin{proof} For $\lambda^+\in\ \Omega^+(\la)$ and $\lambda^-\in\ \Omega^-(\la)$,
by Theorem \ref{th:r=0}, we obtain 
$$
\frac{\#\mathcal{OT}(\lambda^+,k+2n-1)}{\#\mathcal{OT}(\lambda,k+2n)}=\frac{2n}{k+2n}\frac{f_{\la^+}}{(k+1)f_{\la}}
$$
and
$$
\frac{\#\mathcal{OT}(\lambda^-,k+2n-1)}{\#\mathcal{OT}(\lambda,k+2n)}=\frac{k}{k+2n}\frac{f_{\la^-}}{f_{\la}}.
$$
But it is well known that $\sum_{\lambda^+\in\ \Omega^+(\la)}f_{\la^+}=(k+1)f_\la$  and $\sum_{\lambda^-\in\ \Omega^-(\la)}f_{\la^-}=f_\la$ (for example, see \cite[Lemmas $2.2$ and $2.3$]{hanxiong1}). Then the proof is complete. 
\end{proof}

Let $\mathbb{R}[x,y]$ be the set of polynomials of $x$ and $y$ with real coefficients. For each nonnegative integer $r$,  let $\mathbb{R}_r[x,y]$ be the set of polynomials in $\mathbb{R}[x,y]$ with degrees at most $r$.
\begin{defi}\label{def:psi}
The operator $\Psi : \mathbb{R}[x,y]\rightarrow \mathbb{R}[x,y]$ is defined by 
$$
\Psi\left(A(x,y)\right):=(x+2y+1)A(x,x+2y)-xA(x-1,x+2y-1)-2yA(x+1,x+2y-1)
$$
for each polynomial $A(x,y)\in \mathbb{R}[x,y]$.
\end{defi}

\begin{lem} \label{th:isopolynomial}
Let $r\in \mathbb{N}$. The operator $\Psi$ provides a bijection between $\mathbb{R}_r[x,y]$ and itself.
Furthermore, $A(x,y)$ and $\Psi\left(A(x,y)\right)$ have the same degree and constant term for each polynomial $A(x,y)\in \mathbb{R}[x,y]$.
\end{lem}
\begin{proof}
It is obvious that the operator $\Psi$ is an $\mathbb{R}$-linear map over $\mathbb{R}[x,y]$.
For $0\leq i \leq r$ we have 
	\begin{align}\label{eq:psixy}
 \Psi\left(x^iy^{r-i}\right)&=(x+2y+1)x^i(x+2y)^{r-i}-x(x-1)^i(x+2y-1)^{r-i}-2y(x+1)^i(x+2y-1)^{r-i}\nonumber
\\
		&=(r+i+1)\,x^i(x+2y)^{r-i}-i\,x^{i-1}(x+2y)^{r-i+1}+K_{r,i}(x,y) 
\end{align}
for some polynomial $K_{r,i}(x,y)$ with degree at most $r-1$.
Let $\alpha_{j,i}(x,y):=x^iy^{j-i}$ and $\beta_{j,i}(x,y)=x^i(x+2y)^{j-i}$ for $0\leq i\leq j$.
Then $\{ \alpha_{j,i}: 0\leq i\leq j \leq  r\}$ and $\{ \beta_{j,i}: 0\leq i\leq j \leq  r\}$ form two bases for the $\mathbb{R}$-linear space $\mathbb{R}_r[x,y]$.
	From \eqref{eq:psixy} we have 
\begin{align}\label{eq:inverse}
&\Psi(\alpha_{r,r},\alpha_{r,r-1},\ldots,\alpha_{r,0},\alpha_{r-1,r-1},\alpha_{r-1,r-2},\ldots,\alpha_{r-1,0},\dots,\alpha_{1,1},\alpha_{1,0},\alpha_{0,0})
\\&
=
(\beta_{r,r},\beta_{r,r-1},\ldots,\beta_{r,0},\beta_{r-1,r-1},\beta_{r-1,r-2},\ldots,\beta_{r-1,0},\dots,\beta_{1,1},\beta_{1,0},\beta_{0,0}) \times M_r \nonumber
\end{align}
 where
\begin{equation*}
M_r=
  \begin{bmatrix}
    2r+1 &  &  &  &  &  &  &  &  &  \\
    -r & 2r &  &  &  &  &  &  &  &  \\
     & & \ddots &  &   &  & \resizebox{4mm}{!} 0 &  &   &  \\
     &  &  & r+1 &  &  &  &  &  &  \\
     &  &  &  & 2r-1 &  &  &  &  &  \\
     &  &  &  & -(r-1) & 2r-2 &  &  &  &  \\
		 & \resizebox{5mm}{!} *  &   &  &  &  & \ddots&  &  &   \\
     &  &   &  &  &  &  & r &  &  \\
     &  &   &  &  &  &  &  & \ddots &  \\
    0 & 0 &  \cdots & 0 & 0 & 0 & \cdots & 0 & \cdots  & 1 
  \end{bmatrix}
\end{equation*}
is an invertible lower triangular matrix with diagonal entries $\{ 2r+1,2r,\ldots,r+1,2r-1,2r-2,\ldots,r,\ldots,7,6,5,4,5,4,3,3,2,1\}$.  
For example, we have 
\[
M_2=
  \begin{bmatrix}
    5 & 0 & 0 & 0 & 0 & 0 
    \\
    -2& 4 & 0 & 0 & 0 & 0 
    \\
    0 & -1& 3 & 0 & 0 & 0
    \\
    0& -2 & 0 & 3 & 0 & 0
    \\
    -1& 1 & -1& -1& 2 & 0
    \\
    0 & 0 & 0 & 0 & 0 & 1
  \end{bmatrix}.
\]
Then $\Psi$ must be an isomorphism from the $\mathbb{R}$-linear space $\mathbb{R}_r[x,y]$ to itself.  
It is obvious that $A(x,y)$ and $\Psi\left(A(x,y)\right)$ have the same degree and constant term.
\end{proof}

Now we are ready to give the proof of Theorem \ref{th: main_size}. 
\begin{proof}[Proof of Theorem \ref{th: main_size}]
Let $k=|\la|$ and  $Q(x,y)=\Psi^{-1}(P(x,x+2y))$. Actually we will prove the following identity
\begin{align}\label{eq: main_size}
 \sum_{T \in \mathcal{OT}(\lambda,k+2n)} \mathrm{wt}_P(T) = \# \mathcal{OT}(\lambda,k+2n)\times (k+2n+1)Q(k,k+2n)
\end{align} 
for any $n\in \mathbb{Z}$ and $\la\in \mathbb{P}$ with $k+2n\geq 0$ by induction on $k+2n$.  For each oscillating tableaux 
$$
T = (\lambda^{0},\ldots,\lambda^{k+2n})\in \mathcal{OT}(\lambda,k+2n),
$$ 
	we have $\lambda^{k+2n-1}\in\ \Omega^+(\la) \cup \Omega^-(\la)$. Therefore
\begin{align}
\sum_{T \in \mathcal{OT}(\lambda,k+2n)} \mathrm{wt}_P(T) 
	&= \sum_{\la^+\in\ \Omega^+(\la)} \ \ \sum_{T^+ \in \mathcal{OT}(\lambda^+,k+2n-1)} (\mathrm{wt}_P(T^+)+P(k,k+2n)) \nonumber
\\
	&\quad +\sum_{\la^-\in\ \Omega^-(\la)}\ \ \sum_{T^- \in \mathcal{OT}(\lambda^-,k+2n-1)} (\mathrm{wt}_P(T^-)+P(k,k+2n)) \nonumber 
\\&= \sum_{\la^+\in\ \Omega^+(\la)}\ \ \sum_{T^+ \in \mathcal{OT}(\lambda^+,k+2n-1)} \mathrm{wt}_P(T^+) \nonumber
	\\
	&\quad +\sum_{\la^-\in\ \Omega^-(\la)}\ \ \sum_{T^- \in \mathcal{OT}(\lambda^-,k+2n-1)} \mathrm{wt}_P(T^-) \nonumber
	\\&\quad +P(k,k+2n)\times \# \mathcal{OT}(\lambda,k+2n).\label{eq:sum_wt}
\end{align}
%Without loss of generality we can assume that $P(k,k+2n)=k^an^{r-a}$ for some $0\leq a \leq r$.
  Let
$$
A(x,y):=(x+2y+1)Q(x,x+2y).
$$ 

When $k+2n=0$, it is obvious that
\begin{align}\label{eq:Pxy1}
\sum_{T \in \mathcal{OT}(\lambda,0)} \mathrm{wt}_P(T) = \# \mathcal{OT}(\lambda,0) \times A(k,n)
\end{align}
for any $\la\in \mathbb{P}$ by Lemma \ref{th:isopolynomial}. 
	
    When $k+2n\geq 1$, by the induction hypothesis, identity \eqref{eq:sum_wt}
	becomes
\begin{align*}
\sum_{T \in \mathcal{OT}(\lambda,k+2n)} \mathrm{wt}_P(T) 
&=\sum_{\lambda^+\in\ \Omega^+(\la)}A(k+1,n-1)\# \mathcal{OT}(\lambda^+,k+2n-1)   \\
	&\qquad +\sum_{\lambda^-\in\ \Omega^-(\la)} A(k-1,n)\# \mathcal{OT}(\lambda^-,k+2n-1)
\\
	&\qquad	+P(k,k+2n) \# \mathcal{OT}(\lambda,k+2n).
\end{align*}
When $n<0$, the above summation equals $0$ and thus \eqref{eq: main_size} is true. When $n\geq 0$,
by Lemma~\ref{th:la+-}, the above summation is equal to
\begin{align}\label{eq: PA1}
& \#\mathcal{OT}(\lambda,k+2n)\Bigl(\frac{2n}{k+2n}A(k+1,n-1)
+\frac{k}{k+2n}A(k-1,n)+P(k,k+2n)\Bigr).
\end{align}
By the definition of $\Psi$, we have 
\begin{align*}
&\ \  P(x,x+2y)=\Psi(Q(x,y))\\&=(x+2y+1)Q(x,x+2y)-xQ(x-1,x+2y-1)-2yQ(x+1,x+2y-1)\\
& =A(x,y)-\frac{x}{x+2y}A(x-1,y)-\frac{2y}{x+2y}A(x+1,y-1).
\end{align*}
Hence, \eqref{eq: PA1} equals 
$\# \mathcal{OT}(\lambda,k+2n) \times A(k,n)$,
which completes the proof.
\end{proof}

By Theorem \ref{th: main_size}, to evaluate the average of the weight function $P(x,y)$ for oscillating tableaux, we only need to calculate 
$\Psi^{-1}(P(x,x+2y))$. However it seems that the inverse of $\Psi$
has no explicit formula for general polynomial $P(x,y)$, except for
some special cases.

\begin{ex}\label{ex:yr}
	Let $P(x,y)=\binom{y}{r}$ where $r\in \mathbb{N}$.  By Definition \ref{def:psi} we obtain
\begin{align*}
	\Psi\left( \binom yr \right)&=(x+2y+1) \binom {x+2y}r-x\binom {x+2y-1}r-2y \binom {x+2y-1}r
	= (r+1)\binom{x+2y}{r}.
\end{align*}
Hence,
	$$Q(x,y)=\Psi^{-1}(P(x,x+2y))=\Psi^{-1}\left(\binom{x+2y}{r}\right)=\frac{1}{r+1}\binom{y}{r}.$$
Therefore by \eqref{eq: Main_size}, 
for any $n,r\in \mathbb{N}$ we have
\begin{align*}%\label{eq:length_formula}
\frac1{\# \mathcal{OT}(\lambda,|\lambda|+2n)} \sum_{T \in \mathcal{OT}(\lambda,|\lambda|+2n)}
\sum_{i=0}^{|\lambda|+2n}\binom{i}{r}
=(|\lambda|+2n+1)\times \frac{1}{r+1} \binom{|\lambda|+2n}{r}.
\end{align*} 
\end{ex}

The above identity can also be derived by direct calculation with the help of 
$$
\sum_{i=0}^{m}\binom{i}{r}=\binom{m+1}{r+1}=\frac{m+1}{r+1}\binom{m}{r}.
$$
Next we give the proofs for Corollaries \ref{th:empty_partition}, \ref{th:r=2} and \ref{th:xy}.

\begin{proof}[Proof of Corollary \ref{th:empty_partition}]
Let $\la=\emptyset$ and $P(x,y)=\binom xr$ in Theorem \ref{th: main_size}. 
	By Lemma~\ref{th:isopolynomial} and $\mathcal{OT}(\emptyset,2n)= (2n-1)!!$ we obtain that
$$  
\frac{2^n n!}{(2n+1)!} \sum_{(\lambda^{0},\la^1,\ldots,\lambda^{2n}) \in \mathcal{OT}(\emptyset,2n)} 
\sum_{i=0}^{2n}\binom{|\la^i|}{r}
$$
is a polynomial $B(n)$ of $n$ with degree at most $r$. It is obvious that $B(n)=0$ for $0\leq n\leq r-1$ and
$$
B(r)=\frac{2^r r!}{(2r+1)!} \sum_{|\la|=r}f_{\la}^2 
\binom{r}{r}
=\frac{2^r r!^2}{(2r+1)!}. 
$$
The last equality is due to the well-known formula (see \cite{Macdonald,stanley2}) $\sum_{|\la|=r}f_{\la}^2 = r!$.   Since $\deg(B(n))\leq r$, we obtain that
$B(n)=\frac{2^r \, r!^2}{(2r+1)!} \binom{n}{r}$. The proof is complete.
\end{proof}

\begin{proof}[Proofs of Corollaries \ref{th:r=2} and \ref{th:xy}]
First we have 
	\begin{align*}
		\Psi(x)&=2x-2y,\\
		\Psi(y)&=2x+4y,\\
		\Psi(xy)&=3x^2+4xy-4y^2-x+2y,\\
		\Psi(x^2)&=3x^2-4xy-x-2y,\\
		\Psi(y^2)&=3x^2+12xy+12y^2-x-2y.
	\end{align*}
	Then by linearity,
    $$\Psi\left(\frac{6x^2+3xy+y^2+2x+3y}{30}\right)=x^2.$$
  In Theorem \ref{th: main_size}, let $P(x,y)=x^2$.  Then 
$$Q(x,y)=\Psi^{-1}(P(x,x+2y))=\Psi^{-1}\left(x^2\right)=\frac{6x^2+3xy+y^2+2x+3y}{30}.$$
Therefore Corollary \ref{th:r=2} holds
by \eqref{eq: Main_size}.
Also we have  
	$$\Psi\left(\frac{3xy+y^2+2x}{12}\right)=x(x+2y).$$
In Theorem \ref{th: main_size}, let $P(x,y)=xy$.  Then 
	$$Q(x,y)=\Psi^{-1}(P(x,x+2y))=\Psi^{-1}\left(x(x+2y)\right)=\frac{3xy+y^2+2x}{12}.$$
By \eqref{eq: Main_size} we obtain Corollary \ref{th:xy}.
\end{proof}

\section{Proofs of the asymptotic formulas}
In this section, we will prove Theorem \ref{th: asy}. 

\begin{proof}[Proof of Theorem \ref{th: asy}]  
Let $I=\{(r,i)\in \mathbb{N}\times\mathbb{N}:i\leq r\}$. For two elements in $I$, we say that  $(r',i')\preceq (r,i)$ if (1) $r'<r$ or (2) $r'=r$ and $i'\leq i$. It is easy to check that '$\preceq$' is a partial order relation on the set $I$.
Let $M_r,\ \alpha_{j,i}$ and $\beta_{j,i}$ be the same as in the proof of Lemma \ref{th:isopolynomial}.
 For $0\leq i\leq  r$, by \eqref{eq:inverse} 
 we obtain
 \begin{align}\label{eq:inverse*}
&
(\beta_{r,r},\beta_{r,r-1},\ldots,\beta_{r,0},\beta_{r-1,r-1},\beta_{r-1,r-2},\ldots,\beta_{r-1,0},\dots,\beta_{1,1},\beta_{1,0},\beta_{0,0})
\\&=
\Psi(\alpha_{r,r},\alpha_{r,r-1},\ldots,\alpha_{r,0},\alpha_{r-1,r-1},\alpha_{r-1,r-2},\ldots,\alpha_{r-1,0},\dots,\alpha_{1,1},\alpha_{1,0},\alpha_{0,0}) \times M_{r}^{-1}
  \nonumber
\end{align}
 where
$M_{r}^{-1}$
is a lower triangular matrix with diagonal entries $\{ (2r+1)^{-1},(2r)^{-1},\ldots,(r+1)^{-1},(2r-1)^{-1},(2r-2)^{-1},\ldots,r^{-1},\ldots,7^{-1},6^{-1},5^{-1},4^{-1},5^{-1},4^{-1},3^{-1},3^{-1},2^{-1},1^{-1}\}$. 
 Then there exist some constants $ m_{r',i'}^{r,i}\in\mathbb{R}$ such that
\begin{align} \label{eq:3.2}
\beta_{r,i}=\Psi\left( \sum_{(r',i')\preceq (r,i)} m_{r',i'}^{r,i} \alpha_{r',i'}\right),
\end{align}
and
\begin{align} 
\label{eq:3.1} m_{r,0}^{r,0}=\frac1{r+1}.
\end{align}

When $i>0$, the identity \eqref{eq:psixy} implies 
\begin{align} \label{eq:3.3}
 \Psi\left(\alpha_{r,i}\right)&
		-(r+i+1)\,\beta_{r,i}+i\,\beta_{r,i-1}\in \mathbb{R}_{r-1}[x,y].
\end{align}
By \eqref{eq:3.2} and \eqref{eq:3.3} we obtain
\begin{align} 
\Psi\left(\alpha_{r,i}-(r+i+1)\sum_{(r',i')\preceq (r,i-1)} m_{r',i'}^{r,i}\alpha_{r',i'} +i\sum_{(r',i')\preceq (r,i-1)} m_{r',i'}^{r,i-1} \alpha_{r',i'} \right)\in \mathbb{R}_{r-1}[x,y],
\end{align}
and therefore
\begin{align} \label{eq:3.6}
\alpha_{r,i}-(r+i+1)\sum_{(r',i')\preceq (r,i-1)} m_{r',i'}^{r,i} \alpha_{r',i'} +i\sum_{(r',i')\preceq (r,i-1)} m_{r',i'}^{r,i-1} \alpha_{r',i'} \in \mathbb{R}_{r-1}[x,y].
\end{align}
	Then the coefficient of $ \alpha_{r,0}$ in \eqref{eq:3.6} must be zero, i.e.,
$$ (r+i+1)\,m^{r,i}_{r,0}=i\,m_{r,0}^{r,i-1}.$$
Finally by \eqref{eq:3.1} and induction we obtain 
\begin{align} \label{eq:3.4} 
m^{r,i}_{r,0}=\frac{r!\, i!}{(r+i+1)!}
\end{align}
for any $0\leq i\leq r$. 
This means that, for a given $k$,  
\begin{align*}
\Psi^{-1}\left(x^i (x+2y)^j\right)\Big|_{x=k, y=2n}
	&=
\Psi^{-1}\left(\beta_{i+j,i}\right)\Big|_{x=k, y=2n}
\\
	&\sim\ 
m^{i+j,i}_{i+j,0}\, (k+2n)^{i+j}
\\ 
	&\sim\ 
\frac {i!(i+j)!(2n)^{i+j}}{(2i+j+1)!}
\end{align*}
when $n\rightarrow \infty$.
Then \eqref{eq: asy2} holds by Theorem \ref{th: main_size}.

\medskip

On the other hand, when $i>0$, by comparing the coefficients of $ \alpha_{r,i'}\ (0\leq i'\leq i)$ in  \eqref{eq:3.6}, we have 
\begin{align} %\label{eq:3.6}
1-(r+i+1)\sum_{i'=0}^{i}  m_{r,i'}^{r,i}  +i\sum_{i'=0}^{i-1}  m_{r,i'}^{r,i-1} =0,
\end{align}
or equivalently, 
\begin{align} %\label{eq:3.6}
\sum_{i'=0}^{i}  m_{r,i'}^{r,i}  = \frac{1+i\sum_{i'=0}^{i-1}  m_{r,i'}^{r,i-1}}{r+i+1}.
\end{align}
Therefore by \eqref{eq:3.1} and induction we obtain 
\begin{align} \label{eq:3.7} 
\sum_{i'=0}^{i}  m_{r,i'}^{r,i}  =\frac{1}{r+1}
\end{align}
for any $0\leq i\leq r$. 
This means that, for a given $n$, 
\begin{align*}
\Psi^{-1}\left(x^i (x+2y)^j\right)\Big|_{x=k, y=2n}
	&=
\Psi^{-1}\left(\beta_{i+j,i}\right)\Big|_{x=k, y=2n}
\\
	&\sim\ 
\sum_{i'=0}^{i} m^{i+j,i}_{i+j,i'}\,k^{i'}\, (k+2n)^{i+j-i'}
\\
	&\sim\ 
\frac {k^{i+j}}{(i+j+1)}
\end{align*}
when $k\rightarrow \infty$.
Then \eqref{eq: asy1} holds by Theorem \ref{th: main_size}.
\end{proof}

\section{Remarks and discussions} 

(1) The operator $\Psi$ defined in Section \ref{sec:size} and its inverse $\Psi^{-1}$ can be realized via a computer algorithm, as shown by the next program 
written for the computer algebra system {\tt Sage} \cite{sage-combinat}.

\smallskip
\hrule
\begin{verbatim}
var('y')
def Psi(P):
  z=x+2*y
  def A(a,b): return P.subs(x=a).subs(y=b)
  return expand((z+1)*A(x,z)-x*A(x-1,z-1)-2*y*A(x+1,z-1))
def InvPsi(P):
  def Cf(P,ij): return P.coefficient(x,ij[0]).coefficient(y,ij[1])
  d=PolynomialRing(RR, 'x,y')(P).total_degree()+1
  V=[(i,j) for i in range(d) for j in range(d) if i+j<d]
  M=[[Cf(Psi(x^a*y^c), b) for b in V] for (a,c) in V]
  R=1/matrix(M)*matrix([[x^a*y^c] for (a,c) in V])
  return sum([Cf(P,V[j])*R[j][0]  for j in range(len(V))])
\end{verbatim}
\hrule

\medskip

Here are two examples to verify the functions {\tt Psi} and {\tt InvPsi}: 

\begin{verbatim}
     sage: Psi(x^2)
       3*x^2 - 4*x*y - x - 2*y
     sage: InvPsi(x^2+2*x*y)
       1/4*x*y + 1/12*y^2 + 1/6*x
\end{verbatim}
\smallskip

(2)  
As noted in Example \ref{ex:yr}, when $P(x,y)=\binom yr$, we have $\Psi^{-1}\left(P(x,x+2y)\right)=\frac{1}{r+1}\binom yr$. So what is $\Psi^{-1}\left(P(x,x+2y)\right)$ when $P(x,y)=\binom xr$? We can not find nice explicit formula for it. Instead we obtain 
\begin{align}\label{eq: xr}
\Psi^{-1}\left(\binom xr\right)\Big|_{x=0, y=2n}= \frac{r!^2\,2^r}{(2r+1)!} \binom{n}{r}
\end{align}
by comparing the identities \eqref{eq: Main_size} and \eqref{eq:size_formula}. 
Is it possible to find a direct proof for \eqref{eq: xr} by the definition of $\Psi$ without Theorem \ref{th: main_size}?

\medskip

(3) The technique developed in the present paper can be used  for studying 
other problems on the averages of weight functions for oscillating tableaux. 
For each box $\square$ in the Young diagram of the partition $\la$, let $h_\square$ and $c_\square$ be its hook length and content respectively (see \cite{Macdonald,stanley2}). 
In a preparing paper,  by applying results from the study of difference operators  on functions of partitions \cite{hanxiong2,hanxiong1,hanxiong3},  
we will establish the following two explicit formulas with very complicated proofs 
for the average
 weights  related to hook lengths and contents:
\begin{align*}%\label{eq: main_hook}
\frac{2^n n!}{(2n+1)!} \sum_{(\lambda^{0},\ldots,\lambda^{l}) \in \mathcal{OT}(\emptyset,2n)} 
\sum_{i=0}^{2n} \sum_{\square\in\lambda^i}\prod_{1\leq j\leq
r}(h_{\square}^2-j^2)
=\frac{(2r)!\,2^r}{(2r+3)(r+1)!} \binom{n}{r+1}
\end{align*}
and
\begin{align*}%\label{eq: main_content}
\frac{2^n n!}{(2n+1)!} \sum_{(\lambda^{0},\ldots,\lambda^{l}) \in \mathcal{OT}(\emptyset,2n)} 
\sum_{i=0}^{2n} \sum_{\square\in\lambda^i}\prod_{0\leq j\leq
r-1}(c_{\square}^2-j^2)
=\frac{r!\,2^r}{(2r+1)(2r+3)} \binom{n}{r+1}.
\end{align*} 
The above two identities can be seen as analogues of the Okada-Panova hook length formula~\cite{panova}  
$$
\frac1{n!}  \sum_{|\lambda|=n} f_\la^2  {\sum_{\square\in\lambda}\prod_{1\leq j\leq
r}(h_{\square}^2-j^2)} = \frac{(2r)!(2r+1)!}{r! (r+1)!^2}\binom{n}{r+1}
$$
and the Fujii-Kanno-Moriyama-Okada content formula \cite{fkmo} 
$$
\frac1{n!}  \sum_{|\lambda|=n} f_\la^2  {\sum_{\square\in\lambda}\prod_{0\leq j\leq
r-1}(c_{\square}^2-j^2)}=
    \frac{(2r)!}{(r+1)!} \binom{n}{r+1}.
$$

The above results suggest that there are several kinds of weight functions of oscillating tableaux whose averages have nice expressions such as polynomials.  We want to summarize known results and find a theory for the averages of more general weight functions of oscillating tableaux in a preparing paper.

\medskip

(4) A \emph{skew oscillating tableau} is a sequence of partitions $T = (\lambda^{0},\lambda^{1},\ldots,\lambda^{l})$ such that $|\la^i/\la^{i+1}|=1$ or $|\la^{i+1}/\la^{i}|=1$ for each $0\leq i \leq l-1$ without the condition that $\la^0=\emptyset$. The enumeration of skew oscillating tableaux was obtained by Roby \cite{roby}. So it is natural to ask the following question: Is there a generalization of Theorem \ref{th: main_size} for skew oscillating tableaux? As pointed out in \cite{HZ14}, the computation suggests that there
are no simple formulas for skew oscillating tableaux.

\medskip

(5) A \emph{strict partition} is a partition whose parts are distinct to each other. We can define  strict oscillating tableaux in a similar way. It will  be interesting to find similar results for the enumerations and average weights for strict oscillating tableaux.

\medskip

(6) As explained in \cite{HZ14}, oscillating
tableaux are closely related to perfect matchings. The results on this topic can be found in \cite{bloom2,chen,chen2,kasraoui}. It will be interesting to apply our results to perfect matchings in  future.

\section*{Acknowledgments}
The second author acknowledges support from Grant P2ZHP2\_171879 of the Swiss National Science Foundation.

\end{document}